\documentclass[12pt]{amsart}
\usepackage{graphicx}
\usepackage{subfigure}
\usepackage{eepic}
\usepackage{xcolor}
\selectcolormodel{gray}
\usepackage{tikz}
\usepackage{amsmath}
\usepackage{a4wide}
\usepackage[utf8]{inputenc}
\usepackage{amssymb}
\usepackage{amsopn}
\usepackage{epsfig}
\usepackage{amsfonts}
\usepackage{latexsym}
\usepackage{amsthm}
\usepackage{enumerate}
\usepackage[UKenglish]{babel}
\usepackage{verbatim}
\usepackage{color}

\allowdisplaybreaks

\newtheorem{theorem}{Theorem}[section]
\newtheorem{lemma}[theorem]{Lemma}
\newtheorem{proposition}[theorem]{Proposition}

\theoremstyle{definition}

\newtheorem{conjecture}[theorem]{Conjecture}
\newtheorem{question}[theorem]{Question}
\theoremstyle{remark}

\numberwithin{equation}{section}

\newcommand{\R}{\ensuremath{\mathbb{R}}}

\newcommand{\N}{\ensuremath{\mathbb{N}}}

\newcommand{\E}{\mathbb{E}}
\newcommand{\F}{\mathcal{F}}

\renewcommand{\O}{\mathcal{O}}

\renewcommand{\L}{\mathcal{L}}

\begin{document}

\title[On the pair correlations of powers of real numbers]{On the pair correlations of powers of real numbers}

\author{Christoph Aistleitner }
\address{Christoph Aistleitner:  Graz University of Technology, Institute of Analysis and Number Theory,	Steyrergasse 30/II,	8010 Graz, Austria}
\email{aistleitner@math.tugraz.at}

\author{Simon Baker}
\address{Simon Baker: School of Mathematics, University of Birmingham, Birmingham,  B15 2TT, UK}
\email{simonbaker412@gmail.com}

\date{}


\begin{abstract}
A classical theorem of Koksma states that for Lebesgue almost every $x>1$ the sequence $(x^n)_{n=1}^{\infty}$ is uniformly distributed modulo one. In the present paper we extend Koksma's theorem to the pair correlation setting. More precisely, we show that for Lebesgue almost every $x>1$ the pair correlations of the fractional parts of $(x^n)_{n=1}^{\infty}$ are asymptotically Poissonian. The proof is based on a martingale approximation method.
\end{abstract}
   
\maketitle

\section{Introduction}
Let $(x_n)_{n=1}^{\infty}$ be a sequence of real numbers in the unit interval. We say that $(x_n)$ has Poissonian pair correlations if for every $s>0$: $$\lim_{N\to\infty}\frac{\#\{1\leq m\neq n\leq N:\|x_n-x_m\|\leq \frac{s}{N}\}}{N}=2s.$$ Here and throughout $\|\cdot\|$ denotes the distance to the nearest integer, and $\{\cdot\}$ denotes the fractional part. It can be shown that if $(y_n)$ is a sequence whose entries are chosen from $[0,1]$ independently and with respect to the uniform probability measure, then almost surely $(y_n)$ has Poissonian pair correlations. Accordingly, one can view the property of having Poissonian pair correlations as an indication of the ``random'' behaviour of a sequence modulo one. 

Part of the motivation behind studying whether a sequence has Poissonian pair correlations comes from a connection with quantum physics. The Berry-Tabor conjecture states that the discrete energy spectrum of a quantum system has Poissonian pair correlations except for in certain degenerate cases. See \cite{mark} for the background in theoretical physics. For some special quantum systems it can be shown that there exists $\alpha\in \R,$ and a sequence of positive integers $(a_n),$ such that the local distribution properties of the discrete energy spectrum of the system agree with those of the sequence $(\{a_n\alpha\})$. This connection lead to several important works of Rudnick, Sarnak, and Zaharescu \cite{RS,RSZ,RZ}. Interestingly the sequence $(n \alpha)$ does not have Poissonian pair correlations for any $\alpha\in \R$. This can be viewed as a consequence of the three gap theorem. See \cite{LS2} for details.

From a number-theoretic perspective it is an interesting and challenging problem to determine whether a sequence $(x_n)$ has Poissonian pair correlations or not. However, there exist only very few positive results in this direction. For example, it is conjectured that $(\{n^2 \alpha\})$ has Poissonian pair correlations, provided that $\alpha$ cannot be approximated very well by rationals. There are partial results in this direction (see \cite{HB,True}), but it seems that the full conjecture is wide out of reach. In contrast, the situation is relatively well-understood from a metric point of view. Let $(a_n)$ be a sequence of distinct positive integers. Then for Lebesgue almost every $\alpha$ the sequence $(\{a_n \alpha\})$ has Poissonian pair correlations, unless the so-called additive energy of the sequence $(a_n)$ is exceptionally large. See \cite{ALL} and \cite{BCGW}. 

In this paper we are interested in the distribution of fractional parts of sequences of the form $(x^n)$, for $x>1$. The study of this family of sequences dates back to work of Hardy \cite{Hardy} who proved that if $x$ is an algebraic number, then $\lim_{n\to\infty}\|x^n\|=0$ if and only if $x$ is a Pisot number. This result was obtained independently by Pisot in \cite{Pisot2}. Pisot had proved in an earlier work \cite{Pisot} that there are at most countably many $x>1$ satisfying $\lim_{n\to\infty}\|x^n\|=0$ . Determining whether there exists transcendental $x>1$ satisfying $\lim_{n\to\infty}\|x^n\|=0$ is still an important open problem. The first metric results on the distribution of the sequence $(x^n)$ were due to Koksma. In \cite{Koks} he proved that for Lebesgue almost every $x>1$ the sequence $(x^n)$ is uniformly distributed modulo one. A version of this theorem for self-similar measures was recently established by the second author in \cite{Bakss}. In \cite{Koks2}, Koksma proved that for any sequence of real numbers $(y_n)$, under suitable monotonicity conditions on the sequence $(\epsilon_n)$, the set $$\{x>1:\|x^n-y_n\|\leq \epsilon_n\, \textrm{ for infinitely many }n\in\N \}$$ has zero or full Lebesgue measure depending on whether $\sum_{n=1}^{\infty} \epsilon_n$ diverged or converged. For some more recent work on the distribution of the sequence $(x^n)$ we refer the reader to \cite{A,Bakss,Baker,Bak,BLR, BugMos,Dub,Kah}, as well as to \cite[Chapters 2 and 3]{Bug}, and the references therein.

In \cite{Baker} the second author asked the question:
	\begin{question}
		\label{question}
		Is it true that for Lebesgue almost every $x>1$ the sequence $(\{x^n\})_{n=1}^{\infty}$ has Poissonian pair correlations?
	\end{question}
The second author was unable to answer this question. Instead he proved via a second moment argument that if $(d_n)$ is a sequence of integers that satisfies a certain growth condition, then for Lebesgue almost every $x>1$ the sequence $(\{x^{d_n}\})$ has Poissonian pair correlations. For example one can take $(d_n)$ to be $(n^k)_{n=1}^{\infty}$ for any $k\geq 2$. In this paper we obtain a positive answer to the question above. In particular we prove the following theorem.

\begin{theorem}
\label{Main theorem}
For Lebesgue almost every $x>1$ the sequence $(\{x^n\})_{n=1}^{\infty}$ has Poissonian pair correlations.
\end{theorem}

It is known that having Poissonian pair correlations is a stronger property than being equidistributed; see for example \cite{ALP,grep}. Thus our theorem is indeed an extension of Koksma's theorem. A simple adaptation of our proof of Theorem \ref{Main theorem} shows that for any fixed $\xi\neq 0$ the sequence $(\{\xi x^n\})_{n=1}^{\infty}$ has Poissonian pair correlations for Lebesgue almost every $x>1$. For simplicity we restrict our attention to the case $\xi=1$.\\

We give a brief heuristic for the theorem, and outline the strategy of proof. Let us first consider a so-called lacunary sequence of integers $(a_n)$, that is, a sequence of at least exponential growth. It is a classical observation that in this case the functions $(\{a_n \alpha\})$ exhibit properties which are typical for sequences of independent, identically distributed (i.i.d.) random variables. In the context of pair correlations, Rudnick and Zaharescu \cite{RZ2} proved that for a lacunary sequence $(a_n)$, the pair correlations of $(\{a_n \alpha\})$ are Poissonian for Lebesgue almost every $\alpha$. Actually, their results go far beyond pair correlations: they could show that for generic $\alpha$ also the triple correlations and all other higher correlations coincide with the Poissonian model, which implies that the distribution of the level spacings (nearest-neighbour spacings) is asymptotically exponential (exactly as the distribution of spacings in the Poisson process -- this is the source of the term ``Poissonian''). The reason for this ``pseudo-random'' behaviour of the sequence $(\{a_n \alpha\})$ is the oscillatory nature with quickly growing frequencies of these functions (which is exploited in the proof in \cite{RZ2} by reducing the problem to the counting of the number of solutions of certain Diophantine equations). In a similar way, the sequence of functions $(\{x^n\})$ exhibits oscillatory and ``pseudo-random'' behaviour; Donald Knuth mentions such sequences in his celebrated \emph{The art of computer programming} as examples of sequences exhibiting a high degree of pseudo-randomness, and stated a conjecture on their statistical properties which was solved by Niederreiter and Tichy \cite{nt}. From a technical perspective, sequences of the form $(\{x^n\})$ are significantly more difficult to handle than lacunary sequences, since they are lacking the simple ``homogeneous'' nature of lacunary sequences, which by orthogonality of the trigonometric system allows the reduction of moment estimates to a simple counting of solutions of Diophantine equations. To overcome these problems, in the present paper we set up a martingale approximation machinery to prove Theorem \ref{Main theorem}, which makes direct use of the oscillatory nature of the sequence $(\{x^n\})$. This martingale method was introduced in metric number theory for problems concerning lacunary trigonometric sequences independently by Berkes \cite{berk,berk2} and Philipp and Stout \cite{ps}. In the context of pair correlation problems (for the case of lacunary sequences), it was used by Berkes, Philipp and Tichy in \cite{bpt}. Roughly speaking, the functional principle of the method is the observation that when we have two oscillating functions, where the frequency of oscillation of the second function is much higher than that of the first, then we can closely approximate the first function by a step function such that the second function still quickly oscillates on the intervals where the step function is constant -- or, in probabilistic language, the conditional expectation of the second function is essentially independent of the sigma-field under which the discretization of the first function is measurable. This allows to approximate the whole structure by a martingale, and it is known from probability theory that in many respects martingales show the same distributional behaviour as sums of independent random variables. The martingale approximation is technically involved, which prevented us from addressing the case of triple or higher correlations. However, we believe that these correlations should also follow the Poissonian model, for Lebesgue almost every $\alpha$. 

\begin{conjecture}
For Lebesgue almost every $x>1$ the triple correlations of the sequence $(\{x^n\})_{n=1}^{\infty}$ coincide with the Poissonian model. The same is true for all higher correlations, as well as for the level spacings. 
\end{conjecture}

Theorem \ref{Main theorem} is equivalent to the following statement: for Lebesgue almost every $x>1$, for all $s>0$ we have $$\lim_{N\to\infty}\frac{\sum_{1\leq m\neq n\leq N}\chi_{[\frac{-s}{N},\frac{s}{N}]}(x^n-x^m)}{N}=2s.$$ Here, and in the sequel, $\chi_{[\frac{-s}{N},\frac{s}{N}]}$ denotes the indicator function on the set $\mathbb{Z}+[\frac{-s}{N},\frac{s}{N}]$. Instead of studying the function $\chi_{[\frac{-s}{N},\frac{s}{N}]}$ directly it is more convenient to study a suitable smooth approximation. Most of this paper will be concerned with proving the following theorem which considers such approximations.

\begin{theorem}
	\label{Approximating thm}
	Let $s>0$ be fixed. Suppose $(F_N)_{N=1}^{\infty}$ is a sequence of differentiable functions satisfying:
	\begin{enumerate}
		\item $F_{N}(x+1)=F_{N}(x)$ for all $x\in \R$ and $N\in \N$.
		\item $F_{N}(x)=F_{N}(-x)$ for all $x\in \R$ and $N\in \N$. 
		\item 	$$\int_{0}^1 F_{N}\, dx= \frac{2s}{N}+\O\left(\frac{1}{N^2}\right).$$
		\item $0\leq F_{N}(x)\leq 1$ for all $x\in \R$ and $N\in \N$.
		\item $\|F'_N\|_{\infty}= \O(N^2).$
		\item $\text{supp}(F_N)\subseteq \mathbb{Z}+[\frac{-2s}{N},\frac{2s}{N}]$.
	\end{enumerate}
Then Lebesgue almost every $x>1$ satisfies $$\lim_{N\to\infty}\frac{\sum_{1\leq m\neq n\leq N^{20}}F_{N^{20}}(x^n-x^m)}{N^{20}}=2s.$$
\end{theorem}
Once Theorem \ref{Approximating thm} is established Theorem \ref{Main theorem} will follow almost immediately. 

\subsection{Structure of the paper and notational conventions.} In Section \ref{Approximating theorem proof} we prove Theorem \ref{Approximating thm}. In Section \ref{Main theorem proof} we prove Theorem \ref{Main theorem}. Throughout this paper we will use the standard big $\O$ notation. That is we say $X=\O(Y)$ if there exists $C>0$ such that $|X|\leq CY$. We will use $\L(\cdot)$ to denote the Lebesgue measure.


\section{Proof of Theorem \ref{Approximating thm}}
\label{Approximating theorem proof}
We now proceed with our proof of Theorem \ref{Approximating thm}. Throughout this section the parameter $s>0$ will be fixed and $(F_N)_{N=1}^{\infty}$ will be a sequence of functions satisfying the hypothesis of Theorem \ref{Approximating thm}. We also choose a number $A>1$ arbitrarily, and keep its value fixed throughout this section. Since $A$ and $s$ are considered to be fixed, in our proof we will suppress the dependence of the implied $\O$ constants on $A$ and $s.$ 


For each $N\in \N$ let $$G_N:=F_N-\int_{0}^1 F_{N}\, dx.$$ It will be technically more convenient to study the typical behaviour of $$\sum_{1\leq m\neq n\leq N}G_N(x^n-x^m)$$ rather than $$\sum_{1\leq m\neq n\leq N}F_N(x^n-x^m).$$ By our underlying assumptions we know that $F_{N}(x)=F_{N}(-x)$ for all $x\in \R$. Therefore  $G_{N}(x)=G_{N}(-x)$ for all $x\in \R$. This implies 
\begin{equation}
\label{n>m}
\sum_{1\leq m\neq n\leq N}G_N(x^n-x^m)=2\sum_{1\leq m< n\leq N}G_N(x^n-x^m).
\end{equation} As such, in our analysis we may always assume $m<n$.


We now partition the set $\{1,\ldots,N\}$ into blocks which describe the magnitude of the parameter $n$ on an appropriate scale. Suppose $N$ is some large number. As the conclusion of Theorem \ref{Approximating thm} indicates, we will only need to consider values of $N$ which are the 20th power of an integer. Thus we can assume throughout this section (for simplicity of writing) that $N^{1/10}$ is an integer, and that $N^{9/10}$ is an integer as well.\\

For each $1\leq k \leq N^{9/10}$ let $$\Delta_{k,N}:= \left\{(k-1) N^{1/10} +1,\ldots, k N^{1/10} \right\}.$$ Note that obviously we have 
\begin{equation*}
\#\Delta_{k,N} = N^{1/10} \qquad \text{and} \qquad \bigcup_{k=1}^{N^{9/10}} \Delta_{k,N} = \{1, \dots, N\}.
\end{equation*}For each $1\leq k \leq N^{9/10}$ let us define $$Y_{k,N}(x):=\sum_{n\in \Delta_{k,N}}\sum_{1\leq m<n}G_{N}(x^n-x^m).$$ It follows from \eqref{n>m} that
\begin{equation}
\label{parity conditioning}
\sum_{1\leq m\neq  n\leq N}G_N(x^n-x^m)=2\left(\sum_{\stackrel{1 \leq k \leq N^{9/10},}{k \textrm{ is odd}}} Y_{k,N}(x)+\sum_{\stackrel{1 \leq k \leq N^{9/10},}{k \textrm{ is even}}}  Y_{k,N}(x)\right).
\end{equation} Conditioning on the parity of $k$ within \eqref{parity conditioning} will play an important part in our proof of Theorem \ref{Approximating thm}. By adopting this approach we can obtain stronger quantitative information on the independence between the functions $Y_{k,N}$ and $Y_{j,N}$.

The key to proving Theorem \ref{Approximating thm} will be following proposition.

\begin{proposition}
	\label{Slow growth}
	Let $s>0$ and suppose $(F_N)_{N=1}^{\infty}$ is a sequence of functions satisfying the hypothesis of Theorem \ref{Approximating thm}. Then
	$$\int_{A}^{A+1}\left( \sum_{\stackrel{1 \leq k \leq N^{9/10},}{k \textrm{ is odd}}}  Y_{k,N}\right)^2\,dx=\O(N^{11/10}).$$
	The same result holds if the summation is extended over all even (rather than all odd) values of $k$ in the specified range.
\end{proposition} Expanding the bracket within Proposition \ref{Slow growth} we obtain
\begin{align}
\label{Opened bracket}
&\int_{A}^{A+1}\left(\sum_{\stackrel{1 \leq k \leq N^{9/10},}{k \textrm{ is odd}}}  Y_{k,N}\right)^2\,dx\nonumber\\
=&\int_{A}^{A+1}2\sum_{\stackrel{1\leq j<k\leq N^{9/10},}{j,k \textrm{ are odd}}}Y_{j,N}\cdot Y_{k,N}\, dx +\int_{A}^{A+1} \sum_{\stackrel{1 \leq k \leq N^{9/10},}{k \textrm{ is odd}}} Y_{k,N}^2\, dx.
\end{align}A similar equation holds in the even case. Proposition \ref{Slow growth} will be implied by the following two lemmas. 
\begin{lemma}
	\label{Lemma1}
Suppose $k-j\geq 2$. Then $$\int_{A}^{A+1}Y_{j,N}\cdot Y_{k,N}\,dx=\O\left(\frac{\log N}{N^{18/10}}\right).$$
\end{lemma}

\begin{lemma}
	\label{Lemma2}For all $1\leq k \leq N^{9/10}$ we have
	$$\int_{A}^{A+1}Y_{k,N}^2\, dx=\O\left(N^{2/10}\right).$$
\end{lemma}
The first summation on the right hand side of \eqref{Opened bracket} consists of $\O(N^{18/10})$ terms, and the second summation consists of $\O(N^{9/10})$ terms. Therefore, applying the bounds provided by Lemma \ref{Lemma1} to the first summation, and the bounds provided by Lemma \ref{Lemma2} to the second summation, we may conclude that Proposition \ref{Slow growth} holds in the odd case. The even case follows by similar reasoning. Therefore to prove Proposition \ref{Slow growth} it suffices to prove Lemma \ref{Lemma1} and Lemma \ref{Lemma2}. This we do in the next two sections.

\subsection{Proof of Lemma \ref{Lemma1}}
To prove Lemma \ref{Lemma1} we will introduce a collection of step-functions $Z_{k,N}$ that provide good approximations to $Y_{k,N}$. For this collection of functions it will be easier to prove that for $k-j\geq 2$ the integral $\int_{A}^{A+1} Z_{j,N}\cdot Z_{k,N}\, dx$ is small. Before we can define $Z_{k,N}$ we need to introduce the following $\sigma$-algebra.

Given $k\in \N$ and $x\in [A,A+1),$ we define the quantity $\mu_k(x)$ to be the unique integer solution to the inequalities: 
\begin{equation} \label{*-ex} 
2^{\mu_k(x)}\leq x^{(k+1/2) N^{1/10}}<2^{\mu_k(x)+1}.
\end{equation}
Now let us fix $k\in\N$. Let $z_{0,k}=A$ and define $z_{1,k}=A+2^{-\mu_k(A)}.$ Suppose $z_{0,k},\ldots, z_{j,k}$ have been constructed and $z_{j,k}<A+1$, we define $z_{j+1,k}$ to be $z_{j+1,k}=z_{j,k}+2^{-\mu_k(z_{j,k})}.$ We stop at $N_k$ when $z_{N_k,k}=A+1$. Such an $N_k$ must exist because $\mu_k(x)$ is bounded from above, increasing with $x,$ and because at each step in our construction we are adding a negative power of $2$. We define $\mathcal{F}_k$ to be the $\sigma$-algebra generated by the intervals $$\{[z_{i,k},z_{i+1,k}):0\leq i <N_k\}.$$ Note that $\mu_k(x)\leq \mu_{k+1}(x)$ for all $x\in [A,A+1)$. Using this property, and the fact that each $\F_k$ is generated by a collection of intervals whose length is some negative power of $2$, we can conclude that $\F_j$ is a sub $\sigma$-algebra of $\F_k$ for all $j<k$. 

For each $1\leq k\leq N^{9/10}$ we define $Z_{k,N}:[A,A+1) \to\mathbb{R}$ as follows:
$$Z_{k,N}(x)= \frac{1}{\L([z_{i,k},z_{i,k+1}))}\int_{z_{i,k}}^{z_{i,k+1}}Y_{k,N}\,dx \quad \mbox{if } x\in [z_{i,k},z_{i,k+1})\mbox{ for some } 0\leq i<N_k.$$ We remark that in probabilistic language $Z_{k,N}$ is the conditional expectation of $Y_{k,N}$ under $\mathcal{F}_k$. The increasing system of $\sigma$-algebras $(\mathcal{F}_k)_k$ forms a so-called \emph{filtration}, which is the basis for constructing a martingale.\\

The following lemma describes how good an approximation $Z_{k,N}$ is to $Y_{k,N}$.

\begin{lemma}
	\label{Approximation lemma}
	For each $1\leq k\leq N^{9/10} $ we have $$\|Y_{k,N}-Z_{k,N}\|_{\infty} =\O\left(\frac{N^{41/10}}{A^{N^{1/10}/2}}\right).$$
\end{lemma}

Before we turn to the proof of the lemma, let us quickly reflect on what we are doing here. We have split the sum $\sum_{m<n} G(x^n - x^m)$ into block sums $Y_{k,N}$, according to the size of the larger index $n$. We can think of the functions $G(x^n-x^m)$ as being quickly oscillating, with the oscillation frequency growing exponentially as a function of $n$. We defined a $\sigma$-field $\mathcal{F}_k$, and replaced $Y_{k,N}$ by the conditional expectation of $Y_{k,N}$ under $\mathcal{F}_k$. Note that this conditional expectation is a step function, which is $\mathcal{F}_k$-measurable. Using the function $\mu_k$ from \eqref{*-ex} we constructed $\mathcal{F}_k$ in an ``inhomogeneous'' way, in the sense that its atoms become finer and finer when moving from $A$ towards $A+1$, which reflects the fact that the ``oscillation frequency'' of $G(x^n - x^m)$ is not everywhere the same, but also increases when $x$ moves from $A$ to $A+1$. We have constructed $\mathcal{F}_k$ in such a way that $Z_{k,N}=\mathbb{E}(Y_{k,N}|\mathcal{F}_k)$ is a good approximation to $Y_{k,N}$.\footnote{We write $\mathbb{E}$ for the expectation (integral) with respect to $x$, on the interval $[A,A+1)$ equipped with Borel sets and Lebesgue measure.} However, since $Z_{k,N}$ is constant on the atoms of $\mathcal{F}_k$, we can use the fact that the functions in the next block $Z_{k+2,N}$ oscillate with much higher frequency than those in $Z_{k,N}$, so that they still oscillate quickly on the atoms of $\mathcal{F}_k$. Note that indeed the next block is $Z_{k+2,N}$ and not $Z_{k+1,N}$, since we currently only consider odd values of $k$ -- this is why we split the whole sum into even and odd parts, to separate the oscillation frequencies in different blocks. So from our construction we essentially have $\mathbb{E} (Z_{k+2,N} | \mathcal{F}_k) \approx 0$, and a similar result for any other (odd) index $j>k+2$ instead of $k+2$ (see Lemma \ref{Martingale lemma} below). Since by construction $Z_{k,N}$ is $\mathcal{F}_k$-measurable, we thus have $\mathbb{E}(Z_{k,N} Z_{j,N}) = \mathbb{E} \left( \mathbb{E} (Z_{k,N} Z_{j,N} | \mathcal{F}_k) \right) = \mathbb{E} \left( Z_{k,N} \mathbb{E} (Z_{j,N} | \mathcal{F}_k) \right) \approx 0.$ After changing back from $Z$ to $Y$, this is essentially Lemma \ref{Lemma1}. The proof of Lemma \ref{Lemma2} is obtained in a very similar way.

\begin{proof}[Proof of Lemma \ref{Approximation lemma}]
	Let $x\in[A,A+1)$. Suppose $0\leq i<N_k$ is such that $x\in [z_{i,k},z_{i+1,k}).$ Then it is a consequence of $Y_{k,N}$ being a continuous function that there exists $y\in[z_{i,k},z_{i+1,k}]$ such that $Z_{k,N}(x)=Y_{k,N}(y)$.  This observation will allow us to use the mean value theorem to bound $|Y_{k,N}(x)-Z_{k,N}(x)|$. To prove our lemma we will also require the following estimates:\\
	\renewcommand{\labelenumi}{\alph{enumi})} \begin{enumerate}
		\item For each $1\leq k\leq N^{9/10}$ and $x\in [A,A+1)$ we have 
		\begin{equation}
		\label{Chain rule}
		|Y_{k,N}'(x)|=\O\left( N^{41/10}\cdot x^{k N^{1/10}}\right).
		\end{equation}This bound follows by applying the chain rule to each term in the summation for $Y_{k,N}$ together with the assumption $\|F_N'\|_{\infty}= \O(N^2).$\\
		
		\item For each $1\leq k\leq N^{9/10}$ we have \begin{equation}
		\label{constant}
		\frac{k N^{1/10}}{z_{i,k}2^{\mu_k(z_{i,k})}}=\O(1).
		\end{equation}This follows from the definition of the function $\mu_{k}$.\\
		
		\item For all $x\in \mathbb{R}$ and $r>0$ we have
		\begin{equation}
		\label{well known}
		(1+x)^r\leq e^{rx}.
		\end{equation}
		~\\
	\end{enumerate}
	
	We now apply these bounds together with the mean value theorem:
	\begin{align*}
	|Y_{k,N}(x)-Z_{k,N}(x)|&=|Y_{k,N}(x)-Y_{k,N}(y)|\\
	&\leq |x-y|\sup_{w\in [z_{i,k},z_{i+1,k}]}|Y_{k,N}'(w)|\\
	&\stackrel{\eqref{Chain rule}}{=}\O\left( 2^{-\mu_k(z_{i,k})}N^{41/10}(z_{i,k}+ 2^{-\mu_k(z_{i,k})})^{k N^{1/10} }\right)\\
	&=\O\left(  2^{-\mu_k(z_{i,k})}N^{41/10}z_{i,k}^{k N^{1/10}}\left(1+\frac{1}{z_{i,k}2^{\mu_k(z_{i,k})}}\right)^{k N^{1/10}}\right)\\
	&\stackrel{\eqref{well known}}{=}\O\left( 2^{-\mu_k(z_{i,k})}N^{41/10}z_{i,k}^{k N^{1/10}}\exp\left(\frac{k N^{1/10}}{z_{i,k}2^{\mu_k(z_{i,k})}}\right) \right)\\
	&\stackrel{\eqref{constant}}{=} \O(2^{-\mu_k(z_{i,k})}N^{41/10}z_{i,k}^{k N^{1/10}})\\
	&=\O\left(\frac{z_{i,k}^{k N^{1/10}}N^{41/10}}{z_{i,k}^{(k+1/2)  N^{1/10}}}\right)\\
	&=\O\left(\frac{N^{41/10}}{z_{i,k}^{ N^{1/10}/2}}\right)\\
	&=\O\left(\frac{N^{41/10}}{A^{N^{1/10}/2}}\right).
	\end{align*}
In the final line we used that $A\leq z_{i,k}$ for all $0\leq i< N_k$ and $1\leq k\leq N^{9/10}$. Since $x$ was arbitrary our result follows. 
	
\end{proof}
\begin{lemma}
	\label{Martingale lemma}
	Let $j,k\in \N$ be such that $k-j\geq2$ and let $0\leq i<N_j.$ Then $$\frac{1}{\L([z_{i,j},z_{i+1,j}))}\int_{z_{i,j}}^{z_{i+1,j}}Y_{k,N}\, dx=\O\left(\frac{\log N}{N^{29/10}}\right).$$
Thus we have
$$\mathbb{E}(Z_{k,N}|\F_j)(x) = \mathbb{E}(Y_{k,N}|\F_j)(x) =\O\left(\frac{\log N}{N^{29/10}}\right),$$
uniformly for all $x \in [A,A+1)$.
\end{lemma}

To prove Lemma \ref{Martingale lemma}, we will use the van der Corput lemma (see for example \cite[p.\ 15]{kn}). 


\begin{lemma}[van der Corput lemma]
	\label{van der Corput}
	Let $\phi:[a,b]\to\mathbb{R}$ be differentiable. Assume that $\phi'(x)\geq \gamma$ for all $x\in [a,b]$, and $\phi'$ is monotonic on $[a,b]$. Then $$\left|\int_{a}^be^{2\pi i\phi(x)}\, dx\right|\leq \gamma^{-1}.$$
\end{lemma}

\begin{proof}[Proof of Lemma \ref{Martingale lemma}]

We begin by focusing on the first part of our lemma.

Since $G_N$ is a differentiable function satisfying $G_{N}(x+1)=G_{N}(x)$ for all $x\in\mathbb{R},$ we know that it equals its Fourier series, i.e., $$G_{N}(x)=\sum_{l\in\mathbb{Z}}c_{l,N}e^{2\pi ilx}.$$ Where $c_{0,N}=0$ because $\int_{0}^{1} G_N\, dx=0.$ We let $$P_{N}:=\sum_{|l|\leq N^6}c_{l,N}e^{2\pi ilx}$$ and $$R_{N}:=G_N-P_N.$$ By assumption we have  $\int_{0}^1|G_N|\,dx=\O(1)$, which directly implies that
\begin{equation}
\label{Fourier1}
|c_{l,N}|=\O(1),\, l\neq 0.
\end{equation}
We also have the bound 
\begin{equation}
\label{Fourier2}
\|R_{N}\|_{\infty}=\O\left(\frac{\log N}{N^4}\right).
\end{equation}  
The estimate \eqref{Fourier2} is a variant of Jackson's inequality from approximation theory; we have not found a good reference except in Jackson's old book, so for the convenience of the reader we give a derivation of this estimate at the end of the present proof.\\

By an application of the triangle inequality we obtain
\begin{align}
\label{Triangle}
\left|\int_{z_{i,j}}^{z_{i+1,j}} G_{N}(x^n-x^m)\, dx\right|&\leq \sum_{|l|\leq N^6}\left|\int_{z_{i,j}}^{z_{i+1,j}} c_{l,N}e^{2\pi il(x^n-x^m)}\, dx\right|\\ &+\left|\int_{z_{i,j}}^{z_{i+1,j}}R_{N}(x^n-x^m)\, dx\right|.\nonumber
\end{align} Applying \eqref{Fourier2} to the second term on the right hand side of \eqref{Triangle} we have 
\begin{equation}
\label{Bound1}
\left|\int_{z_{i,j}}^{z_{i+1,j}}R_{N}(x^n-x^m)\, dx\right|=\O\left(\frac{\L([z_{i,j},z_{i+1,j}])\log N}{N^4}\right).
\end{equation}

We now apply the van der Corput lemma to the first term on the right hand side of \eqref{Triangle}. Let $\phi_l(x)=l(x^n-x^m)$. For any $x\in [z_{i,j},z_{i+1,j}]$ we have
\begin{eqnarray*}
\phi_{l}'(x) & = & l (n x^{n-1} - m x^{m-1}) \\
& \geq & l n (x^{n-1} - x^{m-1}) \\
& \geq & lnz_{i,j}^{n-1}(1-z_{i,j}^{-1}).
\end{eqnarray*}
One can also easily check that $\phi''_l(x) = l (n(n-1) x^{n-2} - m (m-1) x^{m-2}) > 0$  for all $x\in [z_{i,j},z_{i+1,j}]$, since all real solutions $x$ to the equality $x^{n-m} = m (m-1) / (n(n-1))$ are obviously smaller than 1 (provided that $n \geq 2$). Therefore $\phi_l'$ is monotonic. Applying Lemma \ref{van der Corput} together with \eqref{Fourier1} we see that $$\left|\int_{z_{i,j}}^{z_{i+1,j}} c_{l,N}e^{2\pi il(x^n-x^m)}\, dx\right|=\O\left(\frac{1}{lnz_{i,j}^{n-1}}\right).$$ Which implies the following bound for the first term on the right hand side of \eqref{Triangle}:
\begin{equation}
\label{Bound2}\sum_{|l|\leq N^6}\left|\int_{z_{i,j}}^{z_{i+1,j}} c_{l,N}e^{2\pi il(x^n-x^m)}\, dx\right|=\O\left(\frac{\log N}{nz_{i,j}^{n-1}}\right).
\end{equation}
Using \eqref{Bound1} and \eqref{Bound2} together with the definition of $\mu_j(z_{i,j}),$ the fact $k-j\geq 2,$ and $n\in \Delta_{k,N},$ we obtain:
\begin{align}
\label{Conditional expectation bound}
\left|\frac{1}{\L([z_{i,j},z_{i+1,j}))}\int_{z_{i,j}}^{z_{i+1,j}} G_{N}(x^n-x^m)\, dx\right|&= \O\left(\frac{ z_{i,j}^{(j+1/2) N^{1/10}}\log N}{nz_{i,j}^{n-1}}+\frac{\log N}{N^4}\right)\nonumber\\
&=\O\left(\frac{\log N}{nz_{i,j}^{ N^{1/10}/2}}+\frac{\log N}{N^4}\right)\nonumber\\
&=\O\left(\frac{\log N}{N^4}\right).
\end{align}In the last line we used that $z_{i,j}^{- N^{1/10}/2}\leq A^{- N^{1/10}/2},$ and this upper bound decays to zero faster than any negative power of $N$. Applying \eqref{Conditional expectation bound} to each term in the summation for $Y_{k,N},$ together with the fact that this summation consists of $\O(N^{11/10})$ terms, we obtain: 
$$\frac{1}{\L([z_{i,j},z_{i+1,j}))}\int_{z_{i,j}}^{z_{i+1,j}}Y_{k,N}\, dx=\O\left(\frac{\log N}{N^{29/10}}\right).$$

For the second assertion of the lemma, we just note that by construction $\F_j$ is a sub $\sigma$-algebra of $\F_k$, and consequently $\E(Z_{k,N}|\F_j)=\E(Y_{k,N}|\F_j)$.\\

We conclude the proof of Lemma \ref{Martingale lemma} by verifying \eqref{Fourier2}. Let 
$$
D(x) = \frac{\sin (2 \pi (N^6+ 1/2)x)}{\sin (\pi x)}
$$
be the Dirichlet kernel of order $N^6$. We have 
\begin{eqnarray*}
R_N (x) & = & G_N(x) - \int_0^1 G_N(y) D(x-y)~dy \\
& = & G_N (x) - \int_0^1 (G_N(x) + (G_N(x-y) - G_N(x))) D(y)~dy \\
& = & -\int_0^1 (G_N(x-y) - G_N(x)) D(y)~dy,
\end{eqnarray*}
since $\int D ~dx= 1$. By periodicity, we consequently have
\begin{eqnarray}  \label{split}
|R_N(x)| & \leq & \left| \int_{-1/N^6}^{1/N^6} (G_N(x-y) - G_N(x)) D(y)~dy \right| \label{split1} \\
& & + \left| \int_{1/N^6}^{1-1/N^6} (G_N(x-y) - G_N(x)) D(y)~dy \right|. \label{split2}
\end{eqnarray}
By assumption the derivative of $G_N$ is uniformly bounded by $\mathcal{O}(N^2)$, and thus 
\begin{equation} \label{quot}
|(G_N(x-y) - G_N(x)) D(y)| \leq \left|\frac{G_N(x-y) - G_N(x)}{\sin \pi y} \right| = \mathcal{O} (N^2).
\end{equation}
Accordingly, we have
$$
\left| \int_{-1/N^6}^{1/N^6} (G_N(x-y) - G_N(x)) D(y)~dy \right|= \O\left( \frac{1}{N^4}\right),
$$
which gives the desired bound for the integral in line \eqref{split}. To estimate the second integral, we use integration by parts and obtain
\begin{eqnarray*}
& & \left| \int_{1/N^6}^{1-1/N^6} (G_N(x-y) - G_N(x)) D(y)~dy \right| \label{plugin}\\ 
& = &  \frac{1}{2 \pi (N^6 + 1/2)} \left| \int_{1/N^6}^{1-1/N^6} \left(\frac{G_N(x-y) - G_N(x)}{\sin \pi y}\right)' \cos (2 \pi (N^6+1/2)y)~dy \right| + \O\left(\frac{1}{N^4}\right), \nonumber
\end{eqnarray*}
where we estimated the contribution coming from the boundary terms using \eqref{quot}. Now we have
\begin{eqnarray*}
\left| \left(\frac{G_N(x-y) - G_N(x)}{\sin \pi y}\right)' \right| & \leq & \left| \frac{G_N'(x-y)}{\sin \pi y} \right| + \pi \left| \frac{G_N(x-y) - G_N(x)}{(\sin \pi y)^2} \right| \\
& = & \O\left(\frac{N^2}{\sin \pi y} + \frac{N^2 \min(y,1-y)}{(\sin \pi y)^2} \right)\\
& = & \O\left(\frac{N^2}{\sin \pi y}\right).
\end{eqnarray*}
Thus we obtain
\begin{align*}
\left| \int_{1/N^6}^{1-1/N^6} (G_N(x-y) - G_N(x)) D(y)~dy \right| &=\O\left(  \frac{1}{N^4} \int_{1/N^6}^{1-1/N^6} \frac{1}{\sin \pi y} ~dy\right)+\O\left(\frac{1}{N^4}\right)\\
&=\O\left( \frac{\log N}{N^4}\right).
\end{align*}
This gives the desired bound for the integral on the right-hand side of line \eqref{split2}, and thus establishes \eqref{Fourier2} as desired.
\end{proof}

With Lemma \ref{Approximation lemma} and Lemma \ref{Martingale lemma} we can now prove Lemma \ref{Lemma1}. Before giving our proof we recall some well known properties of conditional expectation.

Let $(X,\mathcal{B},\mu)$ be a probability space.
\begin{itemize}
	\item Suppose  $F,G:X\to\mathbb{R}$ are random variables and $\mathcal{H}$ is a sub $\sigma$-algebra of $\mathcal{B}$. If $F$ is $\mathcal{H}$-measurable then $\E(F\cdot G|\mathcal{H})=F\cdot\E(G|\mathcal{H})$.
	\item Suppose $F:X\to\mathbb{R}$ is a random variable and $\mathcal{H}$ is a sub $\sigma$-algebra of $\mathcal{B}$. Then $$\int_{X}|\E(F|\mathcal{H})|\, d\mu\leq \int_{X}|F|\,d\mu.$$
\end{itemize}

\begin{proof}[Proof of Lemma \ref{Lemma1}]
Let $k,j\in\N$ be such that $k-j\geq 2.$ By the triangle inequality: 
\begin{eqnarray}
& & \left|\int_{A}^{A+1}Y_{j,N}\cdot Y_{k,N}\, dx\right|  \nonumber \\ 
& = & \left|\int_{A}^{A+1}Y_{j,N}\cdot Y_{k,N}- Z_{j,N}\cdot Z_{k,N}\, dx\right|+\left|\int_{A}^{A+1}Z_{j,N}\cdot Z_{k,N}\, dx\right|.  \label{Triangle2}
\end{eqnarray} 

Using the properties of conditional expectation mentioned above, together with Lemma \ref{Martingale lemma}, we see that the following holds:
\begin{align}
\label{Firstbound}
\left|\int_{A}^{A+1}Z_{j,N}\cdot Z_{k,N}\, dx\right| &=\left|\int_{A}^{A+1}\E(Z_{j,N}\cdot Z_{k,N}|\F_j)\, dx\right| \nonumber\\
&=\left|\int_{A}^{A+1}Z_{j,N}\cdot \E(Z_{k,N}|\F_j)\, dx\right|\nonumber\\
&\leq \int_{A}^{A+1}\left|Z_{j,N}\cdot \E(Z_{k,N}|\F_j)\right|\, dx\nonumber\\
&=\O\left(\frac{\log N}{N^{29/10}}\int_{A}^{A+1}|Z_{j,N}|\, dx\right)\nonumber\\
&=\O\left(\frac{\log N}{N^{29/10}}\int_{A}^{A+1}|Y_{j,N}|\, dx\right)\nonumber\\
&=\O\left(\frac{\log N}{N^{18/10}}\right).
\end{align}In the final line we used that $\|Y_{j,N}\|_{\infty}=\O(N^{11/10}).$

Using Lemma \ref{Approximation lemma} and the fact $\|Y_{k,N}\|_{\infty}=\O(N^{11/10}),$ we similarly obtain:
\begin{align}
&\left|\int_{A}^{A+1}Y_{j,N}\cdot Y_{k,N}- Z_{j,N}\cdot Z_{k,N}\, dx\right|\nonumber\\
=&\left|\int_{A}^{A+1}Y_{j,N}\cdot Y_{k,N}- (Y_{j,N}+(Z_{j,N}-Y_{j,N}))\cdot (Y_{k,N}+(Z_{k,N}-Y_{k,N}))\, dx\right|\nonumber\\
=&\left|\int_{A}^{A+1}Y_{k,N}(Z_{j,N}-Y_{j,N})+Y_{j,N}(Z_{k,N}-Y_{k,N}) + (Z_{j,N}-Y_{j,N})(Z_{k,N}-Y_{k,N}) \, dx\right|\nonumber\\
\leq& \int_{A}^{A+1}\left|Y_{k,N}(Z_{j,N}-Y_{j,N})\right|+\left|Y_{j,N}(Z_{k,N}-Y_{k,N}) \right| + \left| (Z_{j,N}-Y_{j,N})(Z_{k,N}-Y_{k,N})\right|\, dx\nonumber\\
=&\O\left(\frac{N^{52/10}}{A^{ N^{1/10} /2}}\right). \nonumber
\end{align}Substituting this bound as well as \eqref{Firstbound} into \eqref{Triangle2}, and using the fact that $A^{- N^{1/10} /2}$ decays to zero faster than any negative power of $N,$ we obtain the desired result.
\end{proof}

\subsection{Proof of Lemma \ref{Lemma2}}
We start our proof by choosing $N_0$ sufficiently large such that for all $m_1\geq N_0,$ if $n>m_1$ and $m_2<m_1$ then 
\begin{equation}
(\lfloor A^{m_1}- A^{m_2}\rfloor-2)^{1-n/m_1}\leq A^{(m_1-n)/2}.
\end{equation}

Applying the triangle inequality for the $L^2$ norm twice we obtain: 
\begin{align}
\label{Triangle L2}
\left(\int_{A}^{A+1}Y_{k,N}^2\, dx\right)^{1/2}&=\left(\int_{A}^{A+1}\left(\sum_{n\in \Delta_{k,N}}\sum_{1\leq m<n}G_N(x^n-x^m)\right)^2\, dx\right)^{1/2}\nonumber\\
&\leq \sum_{n\in \Delta_{k,N}}\left(\int_{A}^{A+1}\left(\sum_{1\leq m<n}G_N(x^n-x^m)\right)^2\,dx\right)^{1/2}\nonumber\\
&\leq \sum_{n\in \Delta_{k,N}}\left(\int_{A}^{A+1}\left(\sum_{\stackrel{1\leq m<n}{m< N_0}}G_N(x^n-x^m)\right)^2\,dx\right)^{1/2}\nonumber\\
&+\sum_{n\in \Delta_{k,N}}\left(\int_{A}^{A+1}\left(\sum_{\stackrel{1\leq m<n}{m\geq N_0}}G_N(x^n-x^m)\right)^2\,dx\right)^{1/2}\nonumber\\
&= \O(N^{1/10})+\sum_{n\in \Delta_{k,N}}\left(\int_{A}^{A+1}\left(\sum_{\stackrel{1\leq m<n}{m\geq N_0}}G_N(x^n-x^m)\right)^2\,dx\right)^{1/2}.
\end{align}In the final line we used that $\|G_{N}\|_{\infty}=\O(1)$ and $\#\Delta_{k,N} = N^{1/10}$. To complete our proof of Lemma \ref{Lemma2} we need to obtain good bounds for $$\sum_{n\in \Delta_{k,N}}\left(\int_{A}^{A+1}\left(\sum_{\stackrel{1\leq m<n}{m\geq N_0}}G_N(x^n-x^m)\right)^2\,dx\right)^{1/2}.$$ Expanding the bracket within the integral we obtain 
$$\sum_{n\in \Delta_{k,N}}\left(\int_{A}^{A+1}\sum_{N_0\leq m_1,m_2<n}G_{N}(x^n-x^{m_1})G_{N}(x^n-x^{m_2})\,dx\right)^{1/2}.$$

Recall that $G_N=F_N-\int_{0}^1 F_N\, dx$. Using this equation, together with the assumptions $0\leq F_N\leq 1$ and $\textrm{supp}(F_N)\subset \mathbb{Z}+[\frac{-2s}{N},\frac{2s}{N}]$, we see that for any $N_0\leq m_1,m_2<n$ we have
\begin{align}
\label{Integral substitution}
&\int_{A}^{A+1}G_{N}(x^n-x^{m_1})G_{N}(x^n-x^{m_2})\,dx \nonumber\\
=&\int_{A}^{A+1} F_{N}(x^n-x^{m_1})F_{N}(x^n-x^{m_2})\,dx-\int_{0}^1 F_N\,dx\int_{A}^{A+1}F_{N}(x^n-x^{m_1})\,dx \nonumber\\
&-\int_{0}^1 F_N\,dx\int_{A}^{A+1}F_{N}(x^n-x^{m_2})\,dx+\left(\int_{0}^1 F_N\,dx\right)^2\nonumber\\
 \leq & \int_{A}^{A+1} F_{N}(x^n-x^{m_1})F_{N}(x^n-x^{m_2})\,dx + \mathcal{O} \left(\frac{1} {N^2} \right). 
\end{align}

To estimate the integral in \eqref{Integral substitution}, we will make regular use of the following lemma from \cite{Baker}, which we have rewritten slightly to suit our purposes. 

\begin{lemma}{\cite[Lemma 2.1.]{Baker}}
	\label{convexity lemma}
	Let $f:[a,b]\to\mathbb{R}$ be a strictly increasing  differentiable convex function and $s>0$. Then for $N$ sufficiently large  $$\L\left(\alpha\in [a,b]:f(\alpha)\in \mathbb{Z}+\left[\frac{-s}{N},\frac{s}{N}\right]\right)\leq \frac{4s(b-a)}{N}+\O\left(\frac{4s}{Nf'(a)}\right).$$
\end{lemma}

\begin{lemma}
	\label{Integral bound}
	Let $n>m$, then
$$\int_{A}^{A+1}F_{N}(x^n-x^{m})^2 \,dx=\O\left(\frac{1}{N}\right)$$
\end{lemma}
\begin{proof}
Recall that $0\leq F_N\leq 1$ and $\textrm{supp}(F_N)\subset \mathbb{Z}+[\frac{-2s}{N},\frac{2s}{N}].$ Therefore $F_N^2$ can be bounded above by the function $\chi_{[\frac{-2s}{N},\frac{2s}{N}]}$. Therefore to prove our statement it suffices to show that $$\int_{A}^{A+1}\chi_{[\frac{-2s}{N},\frac{2s}{N}]}(x^n-x^{m}) \, dx =\O\left(\frac{1}{N}\right).$$
But this follows immediately from Lemma \ref{convexity lemma}.
\end{proof}
The proof of the following lemma uses ideas from \cite{Baker}.
\begin{lemma}
	\label{Integral bound2}
	Let $n>m_1>m_2$ and $m_1\geq N_0,$ then $$\int_{A}^{A+1}F_{N}(x^n-x^{m_1})F_{N}(x^n-x^{m_2}) \,dx=\O\left(\frac{1}{N^2}+\frac{m_1A^{(m_1-n)/2}}{Nn(n-m_1)}\right)$$
\end{lemma}
\begin{proof}
Since $F_N$ is positive and bounded above by $\chi_{[\frac{-2s}{N},\frac{2s}{N}]}(x),$ it suffices to show that  $$ \int_{A}^{A+1}\chi_{[\frac{-2s}{N},\frac{2s}{N}]}(x^n-x^{m_1})\chi_{[\frac{-2s}{N},\frac{2s}{N}]}(x^n-x^{m_2}) \,dx =\O\left(\frac{1}{N^2}+\frac{m_1A^{(m_1-n)/2}}{Nn(n-m_1)}\right).$$ Importantly $$\chi_{[\frac{-2s}{N},\frac{2s}{N}]}(x^n-x^{m_1})\chi_{[\frac{-2s}{N},\frac{2s}{N}]}(x^n-x^{m_2})=1\implies \chi_{[\frac{-4s}{N},\frac{4s}{N}]}(x^{m_1}-x^{m_2})=1.$$ Therefore 
\begin{align*}
&\int_{A}^{A+1}\chi_{[\frac{-2s}{N},\frac{2s}{N}]}(x^n-x^{m_1})\chi_{[\frac{-2s}{N},\frac{2s}{N}]}(x^n-x^{m_2})\, dx\\
\leq & \int_{A}^{A+1}\chi_{[\frac{-2s}{N},\frac{2s}{N}]}(x^n-x^{m_1})\chi_{[\frac{-4s}{N},\frac{4s}{N}]}(x^{m_1}-x^{m_2})\, dx.
\end{align*} It therefore suffices to show that 
\begin{equation}
\label{Suffices to show}
\int_{A}^{A+1}\chi_{[\frac{-2s}{N},\frac{2s}{N}]}(x^n-x^{m_1})\chi_{[\frac{-4s}{N},\frac{4s}{N}]}(x^{m_1}-x^{m_2})\, dx=\O\left(\frac{1}{N^2}+\frac{m_1A^{(m_1-n)/2}}{Nn(n-m_1)}\right).
\end{equation} To each $\lfloor A^{m_1}- A^{m_2}\rfloor\leq M\leq \lceil (A+1)^{m_1}- (A+1)^{m_2}\rceil$ we let $$I_{M}:=\left\{x\in [A,A+1):x^{m_1}-x^{m_2}\in \left[M-\frac{4s}{N},M+\frac{4s}{N}\right]\right\}.$$ Importantly each $I_M$ is an interval and 
\begin{equation}
\label{partition}
\left\{x\in [A,A+1):x^{m_1}-x^{m_2}\in \mathbb{Z}+\left[-\frac{4s}{N},\frac{4s}{N}\right]\right\}=\bigcup_{M=\lfloor A^{m_1}- A^{m_2}\rfloor}^{\lceil (A+1)^{m_1}- (A+1)^{m_2}\rceil}I_M.
\end{equation} Note that if $M\neq M'$ then $I_M\cap I_{M'}$ is either empty or a single endpoint. Therefore 
\begin{align*}
&\int_{A}^{A+1}\chi_{[\frac{-2s}{N},\frac{2s}{N}]}(x^n-x^{m_1})\chi_{[\frac{-4s}{N},\frac{4s}{N}]}(x^{m_1}-x^{m_2})\,dx \\=&\sum_{M=\lfloor A^{m_1}- A^{m_2}\rfloor}^{\lceil (A+1)^{m_1}- (A+1)^{m_2}\rceil}\int_{I_{M}}\chi_{[\frac{-2s}{N},\frac{2s}{N}]}(x^n-x^{m_1}) \,dx.
\end{align*}
We let $C_{M}$ denote the left endpoint of $I_M.$ Applying Lemma \ref{convexity lemma} to each $I_M$ we obtain
\begin{align}
\label{Convexity bounds}
\sum_{M=\lfloor A^{m_1}- A^{m_2}\rfloor}^{\lceil (A+1)^{m_1}- (A+1)^{m_2}\rceil}\int_{I_{M}}\chi_{[\frac{-2s}{N},\frac{2s}{N}]}(x^n-x^{m_1}) \,dx&=\O\left(\sum_{M=\lfloor A^{m_1}- A^{m_2}\rfloor}^{\lceil (A+1)^{m_1}- (A+1)^{m_2}\rceil}\frac{\L(I_M)}{N}\right)\\
&+\O\left(\sum_{M=\lfloor A^{m_1}- A^{m_2}\rfloor}^{\lceil (A+1)^{m_1}- (A+1)^{m_2}\rceil}\frac{1}{NnC_M^{n}}\right)\nonumber.
\end{align}
We focus on each term on the right hand side of \eqref{Convexity bounds} individually. Starting with the second term, a simple analysis yields $C_{M}\geq (M-1)^{1/m_1}.$ Therefore
\begin{align}
\label{BoundA}
\sum_{M=\lfloor A^{m_1}- A^{m_2}\rfloor}^{\lceil (A+1)^{m_1}- (A+1)^{m_2}\rceil}\frac{1}{NnC_M^{n}}&\leq \sum_{M=\lfloor A^{m_1}- A^{m_2}\rfloor}^{\lceil (A+1)^{m_1}- (A+1)^{m_2}\rceil}\frac{1}{Nn(M-1)^{n/m_1}}\nonumber\\
&\leq \int_{\lfloor A^{m_1}- A^{m_2}\rfloor}^{\lceil (A+1)^{m_1}- (A+1)^{m_2}\rceil+1}\frac{1}{Nn(x-2)^{n/m_1}}\, dx\nonumber\\
&=\frac{1}{Nn}\left[\frac{(x-2)^{1-n/m_1}}{1-n/m_1}\right]_{\lfloor A^{m_1}- A^{m_2}\rfloor}^{\lceil (A+1)^{m_1}- (A+1)^{m_2}\rceil+1}\nonumber\\
&\leq \frac{m_1(\lfloor A^{m_1}- A^{m_2}\rfloor-2)^{1-n/m_1}}{Nn(n-m_1)}\nonumber\\
&\leq \frac{m_1A^{(m_1-n)/2}}{Nn(n-m_1)}.
\end{align}In the last line we used our assumption $m_1\geq N_0$.

We now focus on the first term on the right hand side of \eqref{Convexity bounds}. Applying \eqref{partition} and Lemma \ref{convexity lemma} we obtain 
\begin{align}
\label{BoundB}
\sum_{M=\lfloor A^{m_1}- A^{m_2}\rfloor}^{\lceil (A+1)^{m_1}- (A+1)^{m_2}\rceil}\frac{\L(I_M)}{N}&=\frac{\L(\{x\in [A,A+1):x^{m_1}-x^{m_2}\in \mathbb{Z}+[\frac{-4s}{N},\frac{4s}{N}]\})}{N}\nonumber\\ &=\frac{1}{N}\left(\O\left(\frac{1}{N}\right)+\O\left(\frac{1}{Nm_1A^{m_1}}\right)\right)\nonumber\\
&=\O\left(\frac{1}{N^2}\right).
\end{align}Substituting \eqref{BoundA} and \eqref{BoundB} into \eqref{Convexity bounds} we deduce that \eqref{Suffices to show} holds. This completes the proof of Lemma \ref{Integral bound2}.
\end{proof}

\begin{proof}[Proof of Lemma \ref{Lemma2}]
	By \eqref{Triangle L2} and \eqref{Integral substitution} we have 
	\begin{eqnarray*}
	& & \left(\int_{A}^{A+1}Y_{k,N}^2\,dx\right)^{1/2} \\
	& \leq & \sum_{n \in \Delta_{k,N}} \left( \int_{A}^{A+1}\sum_{N_0\leq m_1,m_2<n} \left( F_{N}(x^n-x^{m_1}) F_{N}(x^n-x^{m_2})\, dx + \mathcal{O} \left( \frac{1}{N^2} \right) \right)\right)^{1/2} \\
	& & + \O\left(N^{1/10}\right).
	\end{eqnarray*}
Using Lemma \ref{Integral bound} and Lemma \ref{Integral bound2} we obtain
\begin{align}
\label{Constant bound}
&\left(\int_{A}^{A+1}Y_{k,N}^2\,dx\right)^{1/2}\nonumber\\
 \leq& \sum_{n \in \Delta_{k,N}} \left( \sum_{N_0\leq m_2<m_1<n}  \O\left(\frac{1}{N^2}+\frac{m_1A^{(m_1-n)/2}}{Nn(n-m_1)}\right) +\sum_{m=N_0}^n \O\left(\frac{1}{N}\right)\right)^{1/2}\\
& + \O\left(N^{1/10}\right)\nonumber.
\end{align}
Note that
$$
\sum_{m=N_0}^n\frac{1}{N}\leq 1\textrm{ and }\sum_{N_0\leq m_2<m_1<n}\frac{1}{N^2} \leq 1,
$$
since $n \leq N$. Furthermore, we observe that
\begin{align*}
\sum_{N_0\leq m_2<m_1<n}\frac{m_1A^{(m_1-n)/2}}{Nn(n-m_1)}&=\sum_{m_2=N_0}^{n-2}\sum_{m_1=m_2+1}^{n-1}\frac{m_1A^{(m_1-n)/2}}{Nn(n-m_1)}\\
&\leq \sum_{m_2=N_0}^{n-2}\frac{1}{N}\sum_{m_1=m_2+1}^{n-1}\frac{A^{(m_1-n)/2}}{(n-m_1)}\\
&\leq \sum_{m_2=N_0}^{n-2}\frac{1}{N(A^{1/2}-1)}.\\
&=\O(1).
\end{align*}
In the penultimate inequality we used properties of geometric series, and in the final line we used that $n\leq N$. We have shown that the sum of the terms in the bracket in \eqref{Constant bound} is $\O(1)$. Finally, using this bound together with $\# \Delta_{k,n} =N^{1/10}$ we arrive at
$$
\left(\int_{0}^{1}Y_{k,N}^2\,dx\right)^{1/2} = \mathcal{O} (N^{1/10}),  
$$
which proves Lemma \ref{Lemma2}.
\end{proof}

\subsection{Proof of Theorem \ref{Approximating thm}}

Throughout this proof we will use primed summation signs, such as $\sideset{}{'} \sum_k$, to indicate that the summation is restricted such that it only contains \emph{odd} values of $k$. By Markov's inequality and Proposition \ref{Slow growth} we have
\begin{align*}
&\L\left(x\in[A,A+1):\left| \sideset{}{'}\sum_{{1\leq k\leq N^{9/10}}} Y_{k,N}\right|\geq N^{6/10} \right)\\
=&\L\left(x\in[A,A+1):\left(\sideset{}{'} \sum_{1\leq k\leq  N^{9/10}}Y_{k,N}\right)^2\geq N^{12/10}\right)\\
=&\O\left(\frac{1}{N^{1/10}}\right).
\end{align*}Restricting to $20$th powers this bound implies
$$\L\left(x\in[A,A+1):\left| \sideset{}{'}\sum_{1\leq k\leq N^{18}} Y_{k,N^{20}}\right|\geq N^{12} \right)=\O\left(\frac{1}{N^2}\right).$$ Clearly $$\sum_{N=1}^{\infty}\frac{1}{N^{2}}<\infty.$$ Therefore by the Borel-Cantelli lemma it follows that for Lebesgue almost every $x\in [A,A+1)$ the inequality 
$$\left| \sideset{}{'}\sum_{1\leq k\leq N^{18}} Y_{k,N^{20}}\right|\geq N^{12}$$
holds for at most finitely many $N$. Therefore Lebesgue almost every $x\in [A,A+1)$ satisfies 
\begin{equation}
\label{odd convergence}
\lim_{N\to\infty}\frac{ \sideset{}{'}\sum_{1\leq k\leq N^{18}} Y_{k,N^{20}}(x)}{N^{20}}=0.
\end{equation} 
An analogous result holds if the summation is extended over all even, instead of all odd, values of $k$ in the specified range. It follows from the definition of $Y_{k,N}$ that for Lebesgue almost every $x\in [A,A+1)$ we have $$\lim_{N\to\infty}\frac{\sum_{1\leq m\neq n\leq N^{20}}G_{N^{20}}(x^n-x^m)}{N^{20}}=0.$$ Recall that $G_{N}=F_N-\int_{0}^{1}F_N\, dx$ and $\int_{0}^{1}F_N\, dx=\frac{2s}{N}+\O(N^{-2})$. Using this information in the equation above, we deduce that for Lebesgue almost every $x\in [A,A+1)$ we have $$\lim_{N\to\infty}\frac{\sum_{1\leq m\neq n\leq N^{20}}F_{N^{20}}(x^n-x^m)}{N^{20}}=2s.$$
This proves Theorem \ref{Approximating thm}.

\section{Proof of Theorem \ref{Main theorem}}
\label{Main theorem proof}

	With Theorem \ref{Approximating thm} we can now prove Theorem \ref{Main theorem}. Let $s>0$ be arbitrary. We can define two sequences of differentiable functions $(F_N^{1})_{N=1}^{\infty}$ and $(F_N^{2})_{N=1}^{\infty}$ satisfying the hypotheses of Theorem \ref{Approximating thm}, which also satisfy $F_N^{1}(x)\leq \chi_{[\frac{-s}{N},\frac{s}{N}]}(x)\leq F_N^{2}(x)$ for all $x\in \R$ and $N\in \N$. Theorem \ref{Approximating thm} therefore implies that Lebesgue almost every $x\in [A,A+1)$ satisfies
	\begin{equation}
	\label{Good behaviour}
	\lim_{N\to\infty}\frac{\sum_{1\leq m\neq n\leq N^{20}}\chi_{[\frac{-s}{N^{20}},\frac{s}{N^{20}}]}(x^n-x^m)}{N^{20}}=2s.
	\end{equation} Let $S\subset (0,\infty)$ be a countable and dense subset. Since the parameter $s$ above was arbitrary, for Lebesgue almost every $x\in [A,A+1),$ equation \eqref{Good behaviour} indeed holds for all $s\in S$. Using the density of $S$ and an approximation argument, it follows that for Lebesgue almost every $x\in [A,A+1)$ \eqref{Good behaviour} holds for all $s\in(0,\infty)$ .

	Now we show how to remove the restriction to $20$th powers. To any $N\in \N$ we associate the quantity $M_N\in \N$ defined via the inequalities $$M_N^{20}\leq N<(M_N+1)^{20}.$$ Observe that
	$$
	\frac{(M_N+1)^{20}}{M_N^{20}} \to 1 \qquad \text{as $N \to \infty$}.
	$$
	Let $\epsilon>0$ be arbitrary. From \eqref{Good behaviour} we deduce that for Lebesgue almost every $x\in[A,A+1),$ for every $s>0$, we have
	\begin{align*}
	&\limsup_{N\to\infty}\frac{\sum_{1\leq m\neq n\leq N}\chi_{[\frac{-s}{N},\frac{s}{N}]}(x^n-x^m)}{N}\\
	\leq & \limsup_{N\to\infty}\frac{\sum_{1\leq m\neq n\leq (M_N+1)^{20}}\chi_{[\frac{-s-\epsilon}{(M_N+1)^{20}},\frac{s+\epsilon}{(M_N+1)^{20}}]}(x^n-x^m)}{M_N^{20}}\\
	= & \limsup_{N\to\infty}\frac{(M_N+1)^{20}}{M_N^{20}}\frac{\sum_{1\leq m\neq n\leq (M_N+1)^{20}}\chi_{[\frac{-s-\epsilon}{(M_N+1)^{20}},\frac{s+\epsilon}{(N_{M}+1)^{20}}]}(x^n-x^m)}{(M_N+1)^{20}}\\
	= &2(s+\epsilon).
	\end{align*} Since $\epsilon$ was arbitrary, we see that for Lebesgue almost every $x\in[A,A+1),$ for every $s>0$ we have
	$$\limsup_{N\to\infty}\frac{\sum_{1\leq m\neq n\leq N}\chi_{[\frac{-s}{N},\frac{s}{N}]}(x^n-x^m)}{N}\leq2s.$$ The corresponding lower bound can be obtained analogously. Therefore, for Lebesgue almost every $x\in[A,A+1),$ for every $s>0$ we have $$\lim_{N\to\infty}\frac{\sum_{1\leq m\neq n\leq N}\chi_{[\frac{-s}{N},\frac{s}{N}]}(x^n-x^m)}{N}=2s.$$ Since $A>1$ was arbitrary, we see that for Lebesgue almost every $x>1,$ for every $s>0$ we have 
	$$\lim_{N\to\infty}\frac{\sum_{1\leq m\neq n\leq N}\chi_{[\frac{-s}{N},\frac{s}{N}]}(x^n-x^m)}{N}= 2s.$$
	This completes the proof of Theorem \ref{Main theorem}.\\

\noindent \textbf{Acknowledgements.} The first author is supported by the Austrian Science Fund (FWF), projects F-5512, I-3466 and Y-901.

\end{document}